\documentclass{amsart}
\usepackage{amssymb,amsmath}
\usepackage{amsfonts,amsthm}

\newtheorem*{thm1.3}{\bf Theorem 1.3}
\newtheorem*{thm1.4}{\bf Theorem 1.4}
\newtheorem*{lem2.1}{\bf Lemma 2.1}
\newtheorem*{lem2.2}{\bf Lemma 2.2}
\newtheorem*{lem2.3}{\bf Lemma 2.3}
\newtheorem*{lem2.4}{\bf Lemma 2.4}
\newtheorem*{lem3.1}{\bf Lemma 3.1} 
\newtheorem*{lem3.2}{\bf Lemma 3.2}
\newtheorem*{lem3.3}{\bf Lemma 3.3}
\newtheorem*{lem3.4}{\bf Lemma 3.4}
\newtheorem*{lem3.5}{\bf Lemma 3.5}
\newtheorem*{lem3.6}{\bf Lemma 3.6}
\newtheorem*{lem3.7}{\bf Lemma 3.7}
\newtheorem*{lem3.8}{\bf Lemma 3.8}
\newtheorem*{lem4.1}{\bf Lemma 4.1}
\newtheorem*{lem4.2}{\bf Lemma 4.2}	
\newtheorem*{lem4.3}{\bf Lemma 4.3}
\newtheorem*{lem5.1}{\bf Lemma 5.1}
\newtheorem*{lem5.2}{\bf Lemma 5.2}
\newtheorem*{lem5.3}{\bf Lemma 5.3}
 
\newtheorem{theorem}{Theorem}[section]
\newtheorem{lemma}[theorem]{lemma}
\newtheorem{remark}[theorem]{Remark}
\begin{document}
\begin{center}
\end{center}
\title{On the Riesz means of $\delta_k(n)$}
\author{ Saurabh Kumar Singh}

\address{ Stat-Math Unit,
Indian Statistical Institute, 
203 BT Road,  Kolkata-700108, INDIA.}

\email{skumar.bhu12@gmail.com}

\begin{abstract}
Let $k\geq 1$ be an integer. Let $\delta_k(n)$ denote the maximum divisor of $n$ 
which is co-prime to $k$. We study the error term of 
the general $m$-th Riesz mean of the arithmetical function 
$\delta_k(n)$ for any positive integer $m \ge 1$, namely the error 
term $E_m(x)$ where

\[
\frac{1}{m!}\sum_{n \leq x}\delta_k(n) \left( 1-\frac{n}{x} \right)^m    
= M_{m, k}(x) + E_{m, k}(x).
\]
We establish a non-trivial  upper bound for $\left | E_{m, k} (x) \right |$, for any 
integer $m\geq 1$.   
\end{abstract}
\maketitle


\section{\bf Introduction}

For any fixed positive integer $k$, we define 
\begin{equation}
\delta_k(n)= \max \{d : d\mid n, \ \ (d, k)= 1  \}.
\end{equation} 
Joshi and Vaidya \cite{JV} proved that
\begin{equation}
\sum_{n \leq x} \delta_k(n)= \frac{k}{2 \sigma(k)} x^2 + E_k(x),
\end{equation} with $E_k(x)= O(x)$ and $\sigma(k)= \sum_{d \mid k}d$, when $k$ is a square free positive 
integer.  They also proved that when $k=p$, a prime,
\[
\varliminf_{n \to \infty} \frac{E_p(x)}{x} =-\frac{p}{p+1}, \ \ \ \textrm{ and } \ \ \ \varlimsup_{n \to \infty} \frac{E_p(x)}{x} =\frac{p}{p+1}.
\] 
 It was proved by Maxsein and Herzog \cite{MH} that for any square free positive integer $k$,   
\[
\varliminf_{n \to \infty} \frac{E_k(x)}{x} \leq -\frac{k}{\sigma(k)},  \ \ \ \textrm{ and } \ \ \ \varlimsup_{n \to \infty} \frac{E_k(x)}{x} \geq \frac{k}{\sigma(k)}. 
\] 

Around the same time, Adhikari, Balasubramanian and Sankaranarayanan \cite{ABS} proved the above results 
by a different method. While a tauberian theorem of Hardy-Littlewood and Karamata was used in \cite{MH} to get the asymptotic formula for $\sum_{n\leq x} \gamma_k(n)$, where $\gamma_k(n)$ is defined by the relation
$\delta_k(n)= \gamma_k * I(n)$ where $*$ is the  Dirichlet convolution and $I$ is the  
identity function, the method of \cite{ABS}
consists of averaging over arithmetical progressions. 

For $k\geq 1$ and square free, Harzog and Maxsein  \cite{MH} had also observed that

\[
\limsup_{x \to \infty}\frac{E_k(x)}{x} \leq \frac{1}{2} d(k),
\]
where $d(k)$ denotes the number of divisors of $k$.
Later Adhikari and Balasubramanian \cite{AB} improved this result of Maxsein and Herzog 
by showing that 
\[
\varlimsup_{n \to \infty} \frac{|E_k(x)|}{x} \leq \frac{1}{2} \left(1-\frac{1}{p+1} \right) d(k),  
\] where $p$ denotes the smallest prime dividing $k$. 

Writing
\[
H_k(x)= \sum_{n\leq x} \frac{\delta_k(n)}{n}- \frac{k x}{\sigma(k)},
\]
one observes (see \cite{ABS}) that 
\[
\frac{E_k(x)}{x} = H_k(x) + O(1). 
\]
 In \cite{AK}, more precise upper and lower bounds for the quantities $\varliminf H_k(x) $ and  
$\varlimsup H_k(x)$  were established.  The aim of this article is to study the error term of the general $m $-th Riesz mean related to the arithmetic function $ \delta_k(n)$ for any 
positive integer $m\ge 1$ and $k\geq 1$ (need not be a square free integer). More precisely, we write
\begin{equation}
\frac{1}{m!}\sum_{n \leq x} \delta_k(n) \left( 1-\frac{n}{x} \right)^m   
= M_{m, k}(x) + E_{m, k}(x)
\end{equation}
where $M_{m, k}(x)$ is the main term (exists) and $E_{m, k}(x)$ is the error term of the sum 
under investigation. We prove the following. 

\begin{theorem} \label{main}
Let $x \ge x_0$ where $x_0$ is a sufficiently large positive number and let $c(\eta)= \frac{2}{1-2^{-\eta}}$ for any $\eta > 0$. 
For any integer $m\ge 1$ and for any integer $k \geq 1$, we have

\[
 \frac{1}{m!}\sum_{n \leq x}\delta_k(n) \left( 1-\frac{n}{x} \right)^m    
=  \frac{ x^2}{(m+2)!} \prod_{p\mid k}\frac{p}{p+1}+ E_{m, k}(x),
\] where
\[
E_{1, k}  (x) \ll   k c(1/2)^{\omega(k)} x^{\frac {1}{2}} \log x, 
\] and for $m\geq 2$, we have 
\[
E_{m, k} (x) \ll   k c(\eta)^{\omega(k)} x^\eta
\]
for any small fixed positive constant $\eta$ and the implied constant is independent of $m$.
\end{theorem}  

\section{Notation}
\begin{enumerate}

\item Throughout the paper, $s=\sigma + it$ ; the parameters $T$ and $x$ are 
sufficiently large real numbers and $m$ is an integer $\ge 1$.

\item $\eta, \ \epsilon$ always denote sufficiently  small positive constants. 
\item As usual $\zeta(s)$ denotes the Riemann zeta-function.
\item $k$ is any square free positive integer. 
\end{enumerate}

\section{\bf Some Lemmas}
 Generating function for $\delta_k(n) $ is given by:
 \begin{lemma}
 We have
 \[
 \sum_{n=1}^\infty \frac{\delta_k(n)}{n^s} = \zeta(s-1) G(s), 
 \] where 
 \begin{align*}
  G(s)= \sum_{n=1}^\infty \frac{g(n)}{n^s} = \prod_{p\mid k} \left( \frac{1- \frac{p}{p^s} }{1- \frac{1}{p^s} } \right) \ll k\  c(\eta)^{\omega(k)} ,
 \end{align*} for $\sigma \geq \eta$ and 
 \[
 c(\eta)= \frac{2}{1- 2^{-\eta}}. 
 \]
\end{lemma}
\begin{proof}
We have (see  \cite[equation 2.2]{ABS}),
\begin{align*}
\sum_{n=2}^\infty \frac{\delta_k(n)}{n^s} &= \prod_p\left(1+\frac{\delta_k(p)}{p^s} + \frac{\delta_k(p^2)}{p^{2s} } +\cdots  \right)\\
&= \prod_{p\mid k} \left( 1+ \frac{1}{p^s}+ \frac{1}{p^{2s}} +\cdots \right)\prod_{p\nmid k} \left( 1+ \frac{p}{p^s}+ \frac{p^2}{p^{2s}}+ \cdots \right) \\
&= \zeta(s-1)\prod_{p\mid k}  \frac{1-\frac{1}{p^{s-1}} }{1-\frac{1}{p^s} } :=\zeta(s-1) G(s),
\end{align*}
since
\begin{equation*}
\delta_k(p^m) = 
\begin{cases}
1 \ \ \ \ \  \text{if} \ \ \ \ p\mid k \\
p^m \ \ \ \text{if} \ \ \ \ p\nmid k.
\end{cases}
\end{equation*}
And for $\sigma \geq \eta \ (>0)$, we observe that  
\begin{align*}
|G(s)| = \prod_{p\mid k} \left|  \frac{1-\frac{1}{p^{s-1}} }{1-\frac{1}{p^s} }  \right| \leq \prod_{p\mid k} \frac{1+ p^{1-\eta}}{ 1- \frac{1}{p^\eta}} \leq \prod_{p\mid k} \frac{2 p}{  1- \frac{1}{2^\eta} }
 \leq k c(\eta)^{\omega(k)}.  
\end{align*} 
\end{proof}

\begin{lemma} \label{lineinte}
Let $m$ be an integer $\ge 1$.
Let $c$ and $y$ be any positive real numbers and $T \ge T_0$ where $T_0$ 
is sufficiently large. Then we have,
 
\begin{equation*}
\frac{1}{2\pi i} \int_{c-iT}^{c+i T} \frac{y^s}{s(s+1)\cdots(s+m)} ds=
\begin{cases}
\frac{1}{m!} \left ( 1-\frac{1}{y} \right )^m + 
O \left ( \frac{4^m y^c}{T^m} \right ) \ &{\rm if } \ \ y \ge 1,\\
O \left ( \frac{1}{T^m} \right ) \ \ &{\rm if } \ \ 0 < y \leq 1.
\end{cases}
\end{equation*}
\end{lemma}
\begin{proof}  
See  \cite[Lemma 3.2]{SS} and also \cite[p.31 Theorem B]{I}). 
\end{proof}

\begin{lemma} \label{functional equation}
The Riemann zeta-function  $ \zeta(s)$ is extended as a meromorphic function 
in the whole complex plane ${\mathbb C}$ with a simple pole at $s=1$ and it 
satisfies a functional equation $\zeta(s) = \chi (s) \zeta (1-s)$ where

\[
\chi (s) = \frac {\pi^{-(1-s)/2} \Gamma \left ( \frac {1-s}{2} \right )}{
\pi^{-s/2} \Gamma \left ( \frac {s}{2} \right )}.
\]
Also, in any bounded vertical strip, using Stirling's formula, we have

\[
\chi (s) = \left ( \frac {2\pi}{t} \right )^{\sigma + i t -1/2} \
e^{i \left ( t+ \frac {\pi}{4} \right )} \left( 1+O \left ( t^{-1} \right ) 
\right )
\]
as $|t| \rightarrow \infty$. Thus, in any bounded vertical strip, 

\[
|\chi (s)| \asymp t^{1/2-\sigma} \left ( 1+O \left ( t^{-1} \right ) \right )
\]
as $|t| \rightarrow \infty$. 
\end {lemma}
\begin{proof}  
See  \cite[p.116]{ECT} or   \cite[p.8-12]{IV}.
\end{proof}
\begin{lemma} \label{bound half}
We have  for $t\geq t_0$ (sufficiently large), 
\[
\zeta(\frac{1}{2}+ it ) \ll t^{1/6} (\log t)^{3/2} 
\] and 
\[
\zeta(1+it) \ll \log t. 
\]
\end{lemma}
\begin{proof}
See \cite[page 99, Theorem 5.5]{ECT} and \cite[page 49, Theorem 3.5]{ECT}
\end{proof}

\section{Proof of theorem 1.1}

From Lemma \ref{lineinte}, with $c = 2+ \frac{1}{\log x}$ and writing $F(s) :=  \zeta(s-1)G(s)$, we have

\begin{align}
S := \sum_{n \leq x} \delta_k(n) \left( 1-\frac{n}{x}
\right)^m \notag
&= \frac{1}{2\pi i } \int_{c-i\infty }^{c+i\infty }F(s)
\frac{x^s}{s(s+1)\cdots (s+m)} \ ds \notag\\
&= \frac {1}{2\pi i} \int_{c-iT }^{c+iT } F(s) \frac{x^s}{s(s+1)\cdots(s+m)}ds +
O \left ( \frac{4^m x^c \log x}{T^m} \right ).
\end{align} Note that the tail portion error term in the  above expression is actually
\[
\ll \frac{4^m}{T^m} x^c \sum_{n\leq x} \frac{\delta_k(n)}{n^c} \ll \frac{4^m x^c \log x}{T^m},
\] since $\delta_k(n)\leq n$.

\noindent
{\bf Case 1: } Let $m=1$. We move the line of integration in the above integral to
$\Re s = \frac{1}{2}$ . In the rectangular contour formed by the line segments 
joining the points $ c-iT$, $ c+iT$, $ \frac{1}{2}  +iT$,
$ \frac{1}{2} -iT$ and $ c-iT$ in the anticlockwise order, we 
observe that $s= 2$  is a simple pole of the integrand. Thus we get the main term $\frac{x^2}{(m+2)!} \prod_{p\mid k} \frac{p}{p+1}$ from  the residue coming from the pole  $s=2$.

\noindent
We note that
\begin{align}
&\frac {1}{2\pi i} \int_{c-iT }^{c+iT } F(s) \frac{x^s}{s(s+1)}ds\notag\\
&=  \frac {1}{2\pi i} \left \{ \int_{\frac {1}{2} +iT }^{c+iT } \cdots 
+ \int_{\frac {1}{2} -iT}^{\frac {1}{2} +iT} \cdots 
+  \int_{c-iT }^{\frac {1}{2}-iT } \cdots \right \} 
+ {\rm sum \ of \ the \ residues}. 
\end{align}
\noindent
The left vertical line segment contributes the quantity:

\begin{align}
Q_1 &:= \frac{1}{2\pi}\int_{-T}^{T}
F (1/2  + it)
\frac{x^{1/2  + it}dt}{(-1/2  + it)(1/2 + it)} 
 \ dt \notag\\
&= \frac {1}{2\pi} \left (\int \limits_{|t| \le t_0}
+ \int \limits_{ t_0 < |t| \leq T} \right ) \frac{x^{ \frac {1}{2}  + it} 
\zeta \left ( -\frac {1}{2}  + it \right )  G\left ( \frac {1}{2}  + it 
\right ) \ dt}{ \left (  \frac {1}{2} + it \right )
\left ( \frac {1}{2}  + it \right ) )} \notag\\
& \ll k\ c(1/2)^{\omega(k)} x^{1/2 } + k\ c(1/2)^{\omega(k)} x^{1/2} \int \limits_{t_0 < |t| \leq T}
t^{1/2-(-1/2  )}\left|\zeta(3/2  + it)G\left ( \frac {1}{2}  + it 
\right ) \right|  \frac{dt}{t^2} \notag\\
& \ll  k\ c(1/2)^{\omega(k)}  x^{1/2 }  + k\ c(1/2)^{\omega(k)} x^{1/2 } \int \limits_{t_0 < t\leq T}  \frac{dt}{t}.  \notag\\
 & \ll k\ c(1/2)^{\omega(k)} x^{1/2 } \log T.
\end{align}
\noindent 
Now we will estimate the contributions coming from the upper horizontal line (lower horizontal line is similar).

\noindent
The horizontal lines in total contribute a quantity which is in
absolute value
\begin{align*} 
&\ll \int_{1/2}^{c} \left | \zeta(\sigma -1 + iT) G(\sigma + iT) 
\frac{ x^{\sigma + iT}}{(\sigma + iT)(\sigma+1 + iT) } \right | d\sigma \\
& \ll \left(\int_{1/2}^1 + \int_1^{3/2} + \int_{3/2}^{c}  \right)|\zeta(\sigma -1 + iT) G(\sigma + iT) | \frac{x^\sigma}{T^2} d\sigma\\
&\ll  k\ c(1/2)^{\omega(k)}  \left\{ \left( \int_{1/2}^1    +  \int_1^{3/2} \right)   T^{1/2- \sigma +1} |\zeta (2- \sigma + iT)|  \frac{x^{\sigma}}{T^2} d\sigma    \right. \\
&\hspace{80pt} \left. + \int_{3/2}^{c}  |\zeta(\sigma -1 + iT)|\frac{x^{\sigma}}{T^2} d \sigma \right\} (\text{by Lemma  \ref{functional equation}}) \\
&\ll k\ c(1/2)^{\omega(k)} \left( \frac{x \log T }{T} + \frac{x^{3/2} \log T}{T^{3/2}} + \frac{x^2 \log T}{T^{ 11/6}} \right) ( \text{by Lemma \ref{bound half}}). 
\end{align*}

\noindent
Collecting all the estimates, and taking $T= x^{10}$ we get:
\begin{align}
E_{1,k}(x) &\ll k\ c(1/2)^{\omega(k)}  \left( x^{1/2 } \log T + \frac{x^2}{T} + \frac{x \log T }{T} + \frac{x^{3/2} \log T}{T^{3/2}} + \frac{x^2 \log T}{T^{ 11/6}}  \right) \notag\\
&\ll k\ c(1/2)^{\omega(k)} x^{1/2 } \log x. 
\end{align}

\noindent 
{\bf Case 2: }Let $m\geq 2$. We move the line of integration to $\Re s = \eta \ (>0)$. 

We note that

\begin{align}
&\frac {1}{2\pi i} \int_{c-iT }^{c+iT } F(s) \frac{x^s}{s(s+1)\cdots (s+m)}ds\notag\\
&=  \frac {1}{2\pi i} \left \{ \int_{\delta +iT }^{c+iT } \cdots 
+ \int_{\delta -iT}^{\delta +iT} \cdots 
+  \int_{c-iT }^{\delta-iT } \cdots \right \} 
+ {\rm sum \ of \ the \ residue}. 
\end{align}

\noindent
The left vertical line segment contributes the quantity:

\begin{align}
Q_m &:= \frac{1}{2\pi}\int_{-T}^{T}
F ( \eta + it)
\frac{x^{\eta  + it}dt}{(\eta + it)(\eta +1 + it) \cdots (\eta +m + it)} \ dt \notag\\
&= \frac {1}{2\pi} \left (\int \limits_{|t| \le t_0}
+ \int \limits_{ t_0< |t| \leq T} \right ) \frac{x^{  \eta + it} 
\zeta \left ( \eta -1  + it \right )  G\left ( \eta  + it 
\right ) \ dt}{ (  \eta  + it) 
 ( \eta +1  + it  )  \cdots (\eta + m+ it)} \notag\\
& \ll k\ c(\eta)^{\omega(k)}   x^\eta +  k\ c(\eta)^{\omega(k)} x^\eta \int \limits_{t_0 < |t| \leq T}
t^{1/2-(\eta -1  )}|\zeta(3/2 - \eta  + it)G ( \eta+ it) |  \frac{dt}{t^{m+1}} \notag\\
& \ll   k\ c(\eta)^{\omega(k)}  x^\eta  + k\ c(\eta)^{\omega(k)} x^\eta \int \limits_{t_0 < t\leq T}  \frac{t^{3/2- \eta}}{t^3} \ dt.  \notag \\
&\ll  k\ c(\eta)^{\omega(k)}  \  x^\eta .
\end{align}

\noindent
Now we will estimate the contributions coming from the upper horizontal line (lower horizontal line is similar).

\noindent
The horizontal lines in total contribute a quantity which is in
absolute value
\begin{align*} 
&\ll \int_\eta^{c} \left | \zeta(\sigma -1 + iT) G(\sigma + iT) 
\frac{ x^{\sigma + iT}}{(\sigma + iT)(\sigma+1 + iT) \cdots (\sigma+m + iT) } \right | d\sigma \\
&\ll c(\eta)^{\omega(k)}  k \left(\int_\eta^1 + \int_1^{3/2} + \int_{3/2}^{c}  \right) |\zeta(\sigma -1 + iT)| \frac{x^\sigma}{T^{k+1}} \\
&\ll  k\ c(\eta)^{\omega(k)}  \left\{ \left( \int_{\eta}^{1/2}+\int_{1/2}^1    +  \int_1^{3/2} \right)   T^{1/2- \sigma +1} |\zeta (2- \sigma + iT)|  \frac{x^{\sigma}}{T^{m+1}} d\sigma    \right. \\
& \hspace{85pt} \left. + \int_{3/2}^{c}  |\zeta(\sigma -1 + iT)|\frac{x^{\sigma}}{T^{m+1}} d \sigma \right\} \\
& \ll k\ c(\eta)^{\omega(k)} \left(  \frac{x^{1/2}}{T^{m-1/2 + \eta}}+  \frac{x \log T}{T^m} + \frac{x^{3/2} (\log T)^{3/2}}{T^{m+ 5/6}} + \frac{x^2 (\log T)^{3/2}}{T^{m+ 5/6}} \right)
\end{align*}

\noindent
Collecting all the estimates, and taking $T= x^{10}$, for $m\geq 2$ we get:
\begin{align}
E_{m,k}(x) 
\ll k \ c(\eta)^{\omega(k)}    x^\eta. 
\end{align} This proves  Theorem \ref{main}.

\begin{remark}
For $m\geq 2$ we may try to  move the line of integration slightly  left of vertical line $0$. On the line $\Re s= 0$, the function $G(s)$ has simple poles at the points $s(\ell, p)= \frac{2 \pi i \ell}{ \log p}$ \  $\forall \ell \in \mathbb{Z}$ and for each prime $p\mid k$. let $p_1, p_2, \cdots p_{r_k}$ be the primes dividing $k$. The total contribution from the simple poles at the points $s(\ell, p)= \frac{2 \pi i \ell}{ \log p_j}$ for $1\leq j \leq r_k$ is given by:
\[
M= \sum_{j=1}^{r_k} \sum_{|\ell|\leq \frac{T \log p_j}{2 \pi}} \zeta \left(\frac{2 \pi i \ell}{ \log p_j} -1 \right) \prod_{p_i\neq p_j} \left( 1-\frac{p_i}{p_j^{\frac{2 \pi i \ell}{ \log p_j}}}\right) \frac{x^{\frac{2 \pi i \ell}{ \log p_j}}}{\frac{2 \pi i \ell}{ \log p_j} \left( \frac{2 \pi i \ell}{ \log p_j}+1\right)\cdots \left( \frac{2 \pi i \ell}{ \log p_j}+m\right)}.
\] 
If one establishes  that $M=o\left(x^\eta\right)$, then this will improve the error term. This seems to be really difficult.   
\end{remark}
\begin{remark}
From the Theorem  \ref{main} observe that 
\[
E_{1,k}(x) \ll_\epsilon x^{1/2+ 10\epsilon} 
\] uniformly for $3\leq k \ll x^\epsilon$ since $\omega(k) \ll \frac{\log}{\log \log k}$ for $k \geq 3$. Also  $ E_{m,k}(x) \ll x^{c_1\eta}$
 uniformly for $3\leq k \ll x^\epsilon$, where $c_1$ is effective positive constant. 
\end{remark}

\noindent
{\bf Acknowledgement:} Author would like to thank Prof. S. D. Adhikari and Prof. A. Sankaranarayanan for suggesting the problem and for all fruitful discussion and suggestions. 
{}



\vskip 1mm
\noindent 
\begin{thebibliography}{}


\bibitem{ABS}
S. D. Adhikari, R. Balashubramanian and A. Sankaranarayanan : \emph{On  an error term related to the greatest 
divisor of n which is prime to k },  Indian  J. pure and appl. Math.,   {\bf 19(9)} (1988),  830-841.

\bibitem{AB}
S. D. Adhikari and R. Balasubramanian: \emph{A note on a certain error term.},  Arch. Math., {\bf 56} (1991), 
37-40.  

\bibitem{AK}
S. D. Adhikari, K. Soundararajan: \emph{Towards the exact nature of a certain error term-II },  Arch. Math.,   {\bf 59} (1992),  442-449.


\bibitem{I}
A.E. Ingham : \emph{The distribution of prime numbers}, Cambridge University 
Press (1995). 

\bibitem{IV} A. Ivi\'c : \emph{The Riemann Zeta-Function : Theory and 
Applications}, Dover Publications, Inc, New York.

\bibitem{JV}
V. S. Joshi and A. M. Vaidya : \emph{Topics in Classical Number Theory},
Colloq. Math. Soc. J{\'a}nos Boly{\'a}i. Budapest (Hungary) {\bf 34} (1981).


\bibitem{MH}
T. Maxsein and J. Herzog: \emph{Mathematiches Forschungsinstitut oberwalfach Tagungsbericht }, {\bf 42} (1986).

\bibitem{SS}
A. Sankaranarayanan and S.K. Singh : \emph{On the Riesz mean of} 
$\frac{n}{\phi(n)}$, Hardy-Ramanujan J., {\bf 36} (2013), 08-20. 





\bibitem{ECT}
E.C. Titchmarsh : \emph{The Theory of the Riemann Zeta function}, (revised by
D.R. Heath-Brown), Clarendon Press, Oxford (1986).
\end{thebibliography}
\end{document}